\documentclass{article}

\usepackage[utf8]{inputenc} 
\usepackage{url}            
\usepackage{booktabs}       
\usepackage{amsfonts}       
\usepackage{amssymb}
\usepackage{amsmath}
\usepackage{nicefrac}       
\usepackage{microtype}      
\usepackage{xcolor}         
\usepackage{stmaryrd}
\usepackage{wasysym}
\usepackage{empheq}
\usepackage{hyperref}       
\usepackage{enumitem}

\usepackage[colors]{optsys}

\newcommand{\BestConvexLowerApproximations}[2]{{\cal C}^{#1}_{#2}}
\newcommand{\BestConvexLowerHomogeneous}[2]{{\cal H}^{#1}_{#2}}

\newcommand{\quasinorm}{\nu}
\newcommand{\supportFunction}{{\sigma}}

\def\figures{./}
\usepackage{fullpage}
\title{Best Convex Lower Approximations\\
  of the $l_0$ Pseudonorm  on Unit Balls}

\author{Thomas Bittar and Jean-Philippe Chancelier and Michel De Lara, \\ 
  CERMICS, \'Ecole des Ponts, Marne-la-Vall\'ee, France}

\begin{document}

\maketitle

\begin{abstract}
  Whereas the norm of a vector measures amplitude (and is a 1-homogeneous
  function), sparsity is measured by the 0-homogeneous l0 pseudonorm, which 
  counts the number of nonzero components.
  We propose a family of conjugacies suitable for the analysis of 0-homogeneous
  functions. These conjugacies are derived from couplings between vectors, given
  by their scalar product divided by a 1-homogeneous normalizing factor. 
  With this, we characterize the best convex lower approximation of a 0-homogeneous function on
  the unit ``ball'' of a normalization function (i.e. a norm without the
  requirement of subadditivity). We do the same with the best convex and 1-homogeneous
  lower approximation.  In particular, we provide expressions for the tightest convex lower
  approximation of the l0 pseudonorm on any unit ball, and we show that the
  tightest norm which minorizes the l0 pseudonorm on the unit ball of any
  lp-norm is the l1-norm. We also provide the tightest convex lower convex
  approximation of the l0 pseudonorm on the unit ball of any norm. 
\end{abstract}






\section{Introduction}

The \emph{counting function}, also called \emph{cardinality function}
or \emph{\lzeropseudonorm}, 
counts the number of nonzero components of a vector in~$\RR^d$.
It is used in sparse optimization, either as objective function or in the
constraints, to obtain solutions with few nonzero entries.
However, the mathematical expression of the \lzeropseudonorm\
makes it difficult to handle as such in optimization problems on~$\RR^d$.
This is why most of the literature on sparse optimization
resorts to \emph{substitute} or \emph{surrogate} problems, 
obtained either from \emph{estimates} (inequalities) for the \lzeropseudonorm,
or from \emph{alternative} sparsity-inducing terms
(especially suitable norms) \cite{Argyriou-Foygel-Srebro:2012}.
We follow this approach, but using (and extending) the so-called
\Capra\ (Constant Along Primal RAys) couplings and conjugacies introduced
in~\cite{Chancelier-DeLara:2020_CAPRA_OPTIMIZATION,Chancelier-DeLara:2021_ECAPRA_JCA}. 

The paper is organized as follows.
In Section~\ref{sec:capra}, we introduce a new \Capra-coupling that extends the
definition in~\cite{Chancelier-DeLara:2020_CAPRA_OPTIMIZATION}.
We establish expressions for \Capra-conjugates 
and \Capra-subdifferentials of $0$-homogeneous functions. 
Then, in \S\ref{subsec:cvx_appr_0_hom}, we manage to obtain convex lower
bounds for general 0-homogeneous functions, thanks to this new coupling.
We specialize these results for the
\lzeropseudonorm \ in~\S\ref{subsec:appli_l0} and we obtain, in particular,
that the tightest norm below $\lzero$ on any $\ell_{p}$-unit ball ($p\in
[1,\infty]$) is the $\ell_{1}$-norm, hence justifying the use of the
$\ell_{1}$-norm as a sparsity-inducing term.
We also provide the tightest convex lower convex
approximation of the l0 pseudonorm on the unit ball of any norm.

\section{\Capra-conjugacies for 0-homogeneous functions}
\label{sec:capra}

In~\S\ref{subsec:def}, we recall definitions related to homogeneous functions on~$\RR^d$,
and we introduce a new \Capra-coupling between vectors of~$\RR^d$. This \Capra\
coupling is suited for the analysis of 0-homogeneous functions, for which we
provide the expression of \Capra-conjugates and \Capra-subdifferential
in~\S\ref{subsec:capra_conj}. 

We work on the Euclidean space~$\RR^d$
(where~$d$ is a positive integer), equipped with the scalar product 
\( \proscal{\cdot}{\cdot} \). 
We use the notation \( \ic{j,k}=\na{j, j+1,\ldots,k-1,k} \) for any pair of
integers such that \( j \leq k \). 
As we manipulate functions with values 
in~$\barRR = [-\infty,+\infty] $,
we adopt the Moreau \emph{lower and upper additions} \cite{Moreau:1970} 
that extend the usual addition with 
\( \np{+\infty} \LowPlus \np{-\infty} = \np{-\infty} \LowPlus \np{+\infty} =
-\infty \) or with
\( \np{+\infty} \UppPlus \np{-\infty} = \np{-\infty} \UppPlus \np{+\infty} =
+\infty \).
For any subset \( \Uncertain \subset \RR^d \), 
$\delta_{\Uncertain} : \RR^d \to \barRR $ denotes the \emph{characteristic function} of the
set~$\Uncertain$: 
\( \delta_{\Uncertain}\np{\uncertain} = 0 \) if \( \uncertain \in \Uncertain \),
and \( \delta_{\Uncertain}\np{\uncertain} = +\infty \) 
if \( \uncertain \not\in \Uncertain \).

\subsection{Definitions}
\label{subsec:def}

\begin{definition}
  We say that a function~$\fonctionprimal : \RR^d \to \barRR$
  is
  \begin{subequations}
    \begin{enumerate}[label=\emph{(\ref{de:homogeneous-divers}{\alph*})}, ref=\ref{de:homogeneous-divers}{\alph*},topsep=\parsep,leftmargin=1cm]
    \item
      \label{eq:function_is_0-homogeneous}
      \emph{0-homogeneous} if
      \(
      \fonctionprimal\np{\rho\primal} = \fonctionprimal\np{\primal} 
      \eqsepv \forall \rho \in \RR\setminus\{0\}
      \eqsepv \forall \primal \in \RR^d
      \eqfinv
      \)
    \item \label{eq:function_is_positively_1-homogeneous}
      \emph{positively 1-homogeneous} if
      \(
      \fonctionprimal\np{\rho\primal} = \rho \fonctionprimal\np{\primal} 
      \eqsepv \forall \rho \in \RR_+\setminus\{0\} 
      \eqsepv \forall \primal \in \RR^d
      \eqfinv
      \)
    \item
      \label{eq:function_is_absolutely_1-homogeneous}
      \emph{absolutely 1-homogeneous} if
      \(
      \fonctionprimal\np{\rho\primal} = \abs{\rho} \fonctionprimal\np{\primal} 
      \eqsepv \forall \rho \in \RR\setminus\{0\}
      \eqsepv \forall \primal \in \RR^d
      \eqfinp
      \)
    \end{enumerate}
  \end{subequations}
  \label{de:homogeneous-divers}
\end{definition}

\begin{example}

  An example of 0-homogeneous function is the \emph{pseudonorm} 
  \( \lzero : \RR^d \to \ic{0,d} \)
  defined by
  \begin{equation}
    \lzero\np{\primal} 
    = \textrm{number of nonzero components of } \primal
    \eqsepv
    \forall \primal \in \RR^d
    \eqfinp
    \label{eq:pseudo_norm_l0}  
  \end{equation}
  For any $p\in ]0,\infty[$, we define
  \( \ell_p\np{\primal}  = \bp{\sum_{i=1}^d
    |\primal_i|^p}^{\frac{1}{p}}\),
  as well as \( \ell_\infty\np{\primal} = \max_{i\in\ic{1,d}} |\primal_i|\). 
  All these functions $\ell_p$ are absolutely 1-homogeneous (and therefore also
  positively 1-homogeneous). When $p\in [1,\infty]$, \( \ell_p\) is a convex
  function which is the well-known $\ell_p$-norm~$\norm{\cdot}_{p}$. For $p \in
  ]0, 1[$, $\ell_{p}$ is not convex anymore and is only a \emph{normalization
    function} (see below) as it lacks the subadditivity property.
\end{example}

\begin{definition}
  \label{de:quasinorm}
  A function $\quasinorm : \RR^{d} \to \RR_{+}$ is said to be a
  \emph{normalization function} if it is a nonnegative, absolutely 1-homogeneous
  and such that, for any $\primal \in \RR^{d}$, we have that $\quasinorm\np{\primal} = 0$ if and only if $\primal = 0$.
  We introduce the subsets, abusively called unit ``sphere'' and ``ball'', 
  \begin{equation}
    \TripleNormSphere_{\quasinorm}= 
    \defset{\primal \in \RR^d}{\quasinorm\np{\primal} = 1} 
    \eqsepv
    \TripleNormSphere_{\quasinorm}^{\np{0}}= 
    \TripleNormSphere_{\quasinorm} \cup \na{0}
    \eqsepv
    \TripleNormBall_{\quasinorm} = 
    \defset{\primal \in \RR^d}{\quasinorm\np{\primal} \leq 1} 
    \eqfinv
    \label{eq:quasinorm_unit_sphere}
  \end{equation}
  and the primal \emph{normalization mapping}\footnote{%
    We distinguish the normalization \emph{function} with codomain~$\RR_{+}$ from
    the normalization \emph{mapping} with codomain~$
    \TripleNormSphere_{\quasinorm}^{\np{0}}$. Indeed, 
    adopting usage in mathematics, we follow Serge Lang and use ''function'' only to
    refer to mappings in which the codomain is a set of numbers (i.e. a subset
    of~$\RR$ or $\CC$), and reserve the term mapping for more general codomains.}~$\normalized_{\quasinorm}: \RR^d \to
  \TripleNormSphere_{\quasinorm}^{\np{0}}$ 
  by 
  \begin{equation}
    \normalized_{\quasinorm} : \primal \in \RR^d
    \mapsto
    \begin{cases}
      \frac{\primal}{\quasinorm\np{\primal}}\eqsepv
      & \primal \neq 0 \eqfinv
      \\
      0 \eqsepv
      & \text{ else.} 
    \end{cases}  
    \label{eq:normalization_mapping}
  \end{equation}
\end{definition}
A normalization function satisfies the same properties as a norm except
subadditivity. Thus, the unit ``ball''~$\TripleNormBall_{\quasinorm}$
in~\eqref{eq:quasinorm_unit_sphere} is not necessarily convex.
When $\quasinorm = \TripleNorm{\cdot}$ is a norm on~$\RR^d$
(resp. $\quasinorm=\Norm{\cdot}_{p}$ is the $\ell_{p}$ norm,
for $p \in [1, +\infty]$), we denote by~$\TripleNormBall$ (resp. $\TripleNormBall_{p}$) the unit ball:
\begin{equation}
  \TripleNormBall = 
  \defset{\primal \in \RR^d}{\TripleNorm{\primal} \leq 1} 
  \eqsepv
  \TripleNormBall_{p} = 
  \defset{\primal \in \RR^d}{\Norm{\primal}_{p} \leq 1} 
  \eqfinp
  \label{eq:triplenorm_unit_sphere}
\end{equation}


\paragraph{\Capra-couplings and conjugacies}

We introduce a new coupling, which extends
\cite[Definition~8]{Chancelier-DeLara:2020_CAPRA_OPTIMIZATION}, where the
normalizing factor was a norm, whereas it is more general in the definition
below. 

\begin{definition}
  Let $\quasinorm : \RR^{d} \to \RR_{+}$ be a normalization function. The \emph{constant along primal rays coupling}
  $\CouplingCapra: \RR^d\times\RR^d \to \RR$, or \Capra,
  between $\RR^d$ and itself, associated with $\quasinorm$, is the function
  \begin{equation}
    \CouplingCapra : \np{\primal,\dual} \in \RR^d\times\RR^d
    \mapsto 
    \proscal{\normalized_{\quasinorm}\np{\primal}}{\dual}
    =
    \begin{cases}
      \frac{ \proscal{\primal}{\dual} }{ \quasinorm\np{\primal} }
      \eqsepv
      & \primal \neq 0 \eqfinv
      \\
      0 \eqsepv
      & \text{ else.} 
    \end{cases}
    \label{eq:coupling_CAPRA}
  \end{equation}
\end{definition}
The coupling \Capra\ has the property of being 
constant along primal rays, hence the acronym~\Capra\
(Constant Along Primal RAys). This is a special case of a one-sided linear coupling as introduced in~\cite{Chancelier-DeLara:2021_ECAPRA_JCA}.
We review concepts and notations related to Fenchel-Moreau conjugacies
\cite{Singer:1997,Rubinov:2000,Martinez-Legaz:2005}.
The classical Fenchel conjugacy~$\Fenchelcoupling$ is outlined in
Appendix~\ref{Background_on_the_Fenchel_conjugacy}.

\begin{subequations}
  \begin{definition}
    The \emph{$\CouplingCapra$-Fenchel-Moreau conjugate} of a 
    function \( \fonctionprimal : \RR^d \to \barRR \), 
    with respect to the coupling~$\CouplingCapra$, is
    the function \( \SFM{\fonctionprimal}{\CouplingCapra} : \RR^d  \to \barRR \) 
    defined by
    \begin{equation}
      \SFM{\fonctionprimal}{\CouplingCapra}\np{\dual} = 
      \sup_{\primal \in \RR^d} \Bp{ \CouplingCapra\np{\primal,\dual} 
        \LowPlus \bp{ -\fonctionprimal\np{\primal} } } 
      \eqsepv \forall \dual \in \RR^d
      \eqfinp
      \label{eq:Fenchel-Moreau_conjugate}
    \end{equation}
    With $\Fenchelcoupling'$ the Fenchel conjugate as defined
    in~\eqref{eq:Fenchel_conjugate_reverse}, we denote
    \begin{equation}
      \SFM{\fonctionprimal}{\CouplingCapra\Fenchelcoupling'}
      = \LFMr{ \bp{ \SFM{\fonctionprimal}{\CouplingCapra} } }
      \eqfinp 
      \label{eq:Fenchel-Moreau_conjugate_Fenchel}
    \end{equation}
    The \emph{$\CouplingCapra$-Fenchel-Moreau biconjugate} of a 
    function \( \fonctionprimal : \RR^d  \to \barRR \), 
    with respect to the coupling~$\CouplingCapra$, is
    the function \( \SFMbi{\fonctionprimal}{\CouplingCapra} : \RR^d \to \barRR \) 
    defined by\footnote{With the coupling~$\CouplingCapra$, we associate the
      \emph{reverse coupling} $\CouplingCapra'$ defined by
      $\CouplingCapra'\np{\dual, \primal} = \CouplingCapra\np{\primal, \dual}$
      for $\np{\primal, \dual} \in \RR^{d} \times \RR^{d}$, hence the notation~$
      \SFMbi{\fonctionprimal}{\CouplingCapra}$.} 
    \begin{equation}
      \label{eq:Fenchel-Moreau_biconjugate}
      \SFMbi{\fonctionprimal}{\CouplingCapra}\np{\primal} = 
      \sup_{ \dual \in \RR^d } \Bp{ \CouplingCapra\np{\primal,\dual} 
        \LowPlus \bp{ -\SFM{\fonctionprimal}{\CouplingCapra}\np{\dual} } } 
      \eqsepv \forall \primal \in \RR^d 
      \eqfinp
    \end{equation}
  \end{definition}
  The Fenchel-Moreau biconjugate of a 
  function \( \fonctionprimal : \RR^d  \to \barRR \) satisfies $\SFMbi{\fonctionprimal}{\CouplingCapra} \leq \fonctionprimal$.
  %
\end{subequations}
%
We define the \emph{\Capra-subdifferential} of 
a function \( \fonctionprimal : \RR^d \to \barRR \) 
at~\( \primal \in  \RR^d \) by \cite{Akian-Gaubert-Kolokoltsov:2002}
\begin{align}
  \subdifferential{\CouplingCapra}{\fonctionprimal}\np{\primal} 
  &=
    \defset{ \dual \in \RR^d }{ %
    \SFM{\fonctionprimal}{\CouplingCapra}\np{\dual}=
    \CouplingCapra\np{\primal, \dual} 
    \LowPlus \bp{ -\fonctionprimal\np{\primal} } }
    \label{eq:Capra-subdifferential_b}
\end{align}

\subsection{\Capra -conjugates and subdifferentials of 0-homogeneous functions}
\label{subsec:capra_conj}

We now provide expressions for the 
\Capra -conjugates and subdifferentials of 0-homogeneous functions.

\begin{proposition}
  \label{pr:Fenchel-Moreau_conjugate_0-homogeneous}
  Let $\quasinorm: \RR^{d} \to \RR_{+}$ be a normalization function and
  $\CouplingCapra$ be the associated coupling~\eqref{eq:coupling_CAPRA}.
  For any 0-homogeneous function~$\fonctionprimal : \RR^d \to \barRR$,
  we have that (with $\Fenchelcoupling$ the Fenchel conjugate as defined
    in~\eqref{eq:Fenchel_conjugate}),
  \begin{equation}
    \SFM{\fonctionprimal}{\CouplingCapra}
    =
    \LFM{ \bp{ \fonctionprimal \UppPlus \delta_{\TripleNormBall_{\quasinorm}} } }
    =
    \LFM{ \bp{ \fonctionprimal \UppPlus
        \delta_{ \TripleNormSphere_{\quasinorm}^{\np{0}} } } }
    \eqfinp
    \label{eq:Fenchel-Moreau_conjugate_0-homogeneous}
  \end{equation}
\end{proposition}

\begin{proof}
  For any \( \dual \in \RR^d \), we have that 
  \begin{align*}
    \SFM{\fonctionprimal}{\CouplingCapra}\np{\dual} = 
    \sup_{\primal \in \RR^d} \Bp{ \CouplingCapra\np{\primal,\dual} 
    \LowPlus \bp{ -\fonctionprimal\np{\primal} } } 
    &=
      \sup_{\primal' \in \RR^d} \Bp{ \proscal{\normalized_{\quasinorm}\np{\primal'}}{\dual} 
      \LowPlus \bp{ -\fonctionprimal\bp{\normalized_{\quasinorm}\np{\primal'}} } }
      \intertext{by definition~\eqref{eq:coupling_CAPRA} of
      \( \CouplingCapra\np{\primal,\dual} \), by
      definition~\eqref{eq:normalization_mapping} of the
      normalization mapping~$\normalized_{\quasinorm}$ and by
      0-homogeneity~\eqref{eq:function_is_0-homogeneous} of 
      the function~$\fonctionprimal$ (including the case $\primal=0$)}
    &=
      \sup_{\sphere \in \TripleNormSphere_{\quasinorm}^{\np{0}}} \Bp{ \proscal{\sphere}{\dual} 
      \LowPlus \bp{ -\fonctionprimal\np{\sphere} } }
      \intertext{by definition~\eqref{eq:quasinorm_unit_sphere} of~$\TripleNormSphere_{\quasinorm}^{\np{0}}$
      and as \( \normalized_{\quasinorm}\np{\RR^d} \subset
      \TripleNormSphere_{\quasinorm}^{\np{0}} = \normalized_{\quasinorm}\np{\TripleNormSphere_{\quasinorm}^{\np{0}}} \subset
      \normalized_{\quasinorm}\np{\RR^d} \)
      using the positive homogeneity of $\quasinorm$ and the property that $\quasinorm\np{\primal} \neq 0$ if $\primal \neq 0$}
    &=
      \sup_{\primal \in \RR^d} \bgp{ \proscal{\primal}{\dual} 
      \LowPlus \Bp{ -\bp{\fonctionprimal
      \UppPlus \delta_{ \TripleNormSphere_{\quasinorm}^{\np{0}} } }\np{\primal} } }
    \\     
    &=
      \LFM{ \bp{ \fonctionprimal \UppPlus\delta_{
      \TripleNormSphere_{\quasinorm}^{\np{0}} } } }\np{\dual}
      \tag{by~\eqref{eq:Fenchel_conjugate} }
      \eqfinp 
  \end{align*}
  We have shown that \(  \SFM{\fonctionprimal}{\CouplingCapra}
  =
  \LFM{ \bp{ \fonctionprimal \UppPlus
      \delta_{ \TripleNormSphere_{\quasinorm}^{\np{0}} } } } \).
  We now prove that \( \LFM{ \bp{ \fonctionprimal \UppPlus \delta_{\TripleNormBall_{\quasinorm}} } }
  =          \LFM{ \bp{ \fonctionprimal \UppPlus
      \delta_{ \TripleNormSphere_{\quasinorm}^{\np{0}} } } } \). 

  On the one hand, as \( \TripleNormBall_{\quasinorm} \supset 
  \TripleNormSphere_{\quasinorm}^{\np{0}} \) implies that 
  \( \fonctionprimal \UppPlus \delta_{\TripleNormBall_{\quasinorm}} 
  \leq
  \fonctionprimal \UppPlus
  \delta_{ \TripleNormSphere_{\quasinorm}^{\np{0}}} \),
  we deduce that
  \(    \LFM{ \bp{ \fonctionprimal \UppPlus \delta_{\TripleNormBall_{\quasinorm}} } }
  \geq 
  \LFM{ \bp{ \fonctionprimal \UppPlus
      \delta_{ \TripleNormSphere_{\quasinorm}^{\np{0}} } } }
  \) by taking the order-reversing Fenchel
  conjugate~\eqref{eq:Fenchel_conjugate}.
  On the other hand, we fix \( \dual \in \RR^d \) and we show that,
  for any \( \primal\in\TripleNormBall_{\quasinorm} \), there exists
  \( \sphere\in\TripleNormSphere_{\quasinorm}^{\np{0}} \) such that
  \( \proscal{\primal}{\dual} \leq  \proscal{\sphere}{\dual} \)
  and \( \fonctionprimal\np{\primal}=\fonctionprimal\np{\sphere} \).
  Indeed, it suffices to take \( \sphere=0 \) when \( \primal=0 \),
  \( \sphere=\primal / \quasinorm\np{\primal} \) when \( \primal\neq 0 \)
  and \( \proscal{\primal}{\dual} \geq 0 \),
  and \( \sphere=-\primal / \quasinorm\np{-\primal} \) when \( \primal\neq 0 \)
  and \( \proscal{\primal}{\dual} \leq 0 \)
  (by using \( \quasinorm\np{\primal} = \quasinorm\np{-\primal} \leq 1 \) as $\primal \in \TripleNormBall_{\quasinorm}$ and the
  0-homogeneity~\eqref{eq:function_is_0-homogeneous} of 
  the function~$\fonctionprimal$). 
  We deduce that
  \[
    \LFM{ \bp{ \fonctionprimal \UppPlus \delta_{\TripleNormBall_{\quasinorm}} } }\np{\dual} = 
    \sup_{\primal \in \TripleNormBall_{\quasinorm}} \Bp{ \proscal{\primal}{\dual} 
      \LowPlus \bp{ -\fonctionprimal\np{\primal}} }
    \leq
    \sup_{\sphere \in \TripleNormSphere_{\quasinorm}^{\np{0}}} \Bp{ \proscal{\sphere}{\dual} 
      \LowPlus \bp{ -\fonctionprimal\np{\sphere} } }
    = \LFM{ \bp{ \fonctionprimal \UppPlus
        \delta_{ \TripleNormSphere_{\quasinorm}^{\np{0}} } } }\np{\dual}
    \eqfinp         
  \]
  This ends the proof. 
\end{proof}


Whereas Proposition~\ref{pr:Fenchel-Moreau_conjugate_0-homogeneous} relates the
\Capra-conjugate of 0-homogeneous functions with the classical Fenchel-Moreau
conjugate, in Proposition~\ref{prop:capra_subdiff}, we relate the
\Capra-subdifferential with the well-known Rockafellar-Moreau subdifferential.

\begin{proposition}
  \label{prop:capra_subdiff}
  Let $\fonctionprimal: \RR^{d} \to \barRR$ be any function.
  For all \( \primal \in \RR^d \),
  the \Capra-subdifferential \( \subdifferential{\CouplingCapra}{\fonctionprimal}\np{\primal} \),
  as in~\eqref{eq:Capra-subdifferential_b}, is a closed convex set.
  Moreover, if the function~$\fonctionprimal$ is 0-homogeneous, 
  we have that
  \begin{subequations}
    \begin{empheq}[left = \empheqlbrace]{align}
      &\subdifferential{\CouplingCapra}{\fonctionprimal}\np{0} = \subdifferential{}{\np{\fonctionprimal \UppPlus \delta_{\TripleNormBall_{\quasinorm}}}}\np{0} \eqfinv
      \label{eq:Capra-subdifferential=Rockafellar-Moreau-subdifferential1}\\
      &\sphere \in \TripleNormSphere_{\quasinorm} \text{ and }
      \SFM{\fonctionprimal}{\CouplingCapra\Fenchelcoupling'}\np{\sphere}
      =\fonctionprimal\np{\sphere}
      \implies
      \subdifferential{\CouplingCapra}{\fonctionprimal}\np{\sphere}
      =\subdifferential{}{\SFM{\fonctionprimal}{\CouplingCapra\Fenchelcoupling'}}\np{\sphere}        
      \eqfinv
      \label{eq:Capra-subdifferential=Rockafellar-Moreau-subdifferential2}
    \end{empheq}
  \end{subequations}
  where $\TripleNormBall_{\quasinorm}$ is defined in~\eqref{eq:quasinorm_unit_sphere} for the normalization function $\quasinorm$ that generates the coupling $\CouplingCapra$, and $\subdifferential{}{\np{\fonctionprimal \UppPlus \delta_{\TripleNormBall_{\quasinorm}}}}$ and $\subdifferential{}{\SFM{\fonctionprimal}{\CouplingCapra\Fenchelcoupling'}}$ are (Rockafellar-Moreau) subdifferentials
  as in~\eqref{eq:Rockafellar-Moreau-subdifferential_a}.
\end{proposition}

\begin{proof}
  We prove that \( \subdifferential{\CouplingCapra}{\fonctionprimal}\np{\primal} \) is a closed convex set.
  Let $\primal \in \RR^{d}$ and first suppose that $\fonctionprimal\np{\primal}
  = - \infty$.   Using~\eqref{eq:Fenchel-Moreau_conjugate}
  and~\eqref{eq:Capra-subdifferential_b}, we can check that
  $\subdifferential{\CouplingCapra}{\fonctionprimal}\np{\primal} = \RR^{d}$,
  which is closed and convex.
  In the case where $\fonctionprimal\np{\primal} = +
  \infty$, we have that 
  $\subdifferential{\CouplingCapra}{\fonctionprimal}\np{\primal} = \emptyset$ if
  $\fonctionprimal$ is not identically $+\infty$, and that
  $\subdifferential{\CouplingCapra}{\fonctionprimal}\np{\primal} = \RR^{d}$
  otherwise; in either cases, the \Capra-subdifferential is closed and
  convex.
  Now, suppose that $\fonctionprimal\np{\primal} \in \RR$.
  By definition~\eqref{eq:Fenchel-Moreau_conjugate}
  of \( \SFM{\fonctionprimal}{\CouplingCapra} \),
  the \Capra-subdifferential~\eqref{eq:Capra-subdifferential_b}
  can be written as
  \(
    \subdifferential{\CouplingCapra}{\fonctionprimal}\np{\primal} 
    =
    \bset{ \dual \in \RR^d }{ %
      \SFM{\fonctionprimal}{\CouplingCapra}\np{\dual} \leq 
      \CouplingCapra\np{\primal, \dual} 
      \LowPlus \bp{ -\fonctionprimal\np{\primal} } }
  \)
  where the function
  \( \SFM{\fonctionprimal}{\CouplingCapra} = \LFM{ \bp{ \fonctionprimal \UppPlus
      \delta_{ \TripleNormSphere_{\quasinorm}^{\np{0}} } } }\) is a Fenchel conjugate
  by~\eqref{eq:Fenchel-Moreau_conjugate_0-homogeneous}
  hence is closed convex (see the background material in
  Appendix~\ref{Background_on_the_Fenchel_conjugacy}), 
  and the function~\( g_{\primal}:~\dual \mapsto \CouplingCapra\np{\primal,
    \dual} \LowPlus \bp{ -\fonctionprimal\np{\primal} } = \CouplingCapra\np{\primal,
    \dual} -\fonctionprimal\np{\primal}  \) is affine.
  As a consequence, the set of points where $\SFM{\fonctionprimal}{\CouplingCapra} \leq g_{\primal}$, which is exactly \( \subdifferential{\CouplingCapra}{\fonctionprimal}\np{\primal} \), is a closed convex set.
  \medskip

  We prove~\eqref{eq:Capra-subdifferential=Rockafellar-Moreau-subdifferential1}
  as follows: 
  \begin{align*}
    \dual\in\subdifferential{}{\np{\fonctionprimal \UppPlus \delta_{\TripleNormBall_{\quasinorm}}}}\np{0}
    & \iff
      \LFM{\bp{\fonctionprimal \UppPlus \delta_{\TripleNormBall_{\quasinorm}}}}\np{\dual} 
      =\proscal{0}{\dual} \LowPlus \bp{ -\np{\fonctionprimal
      \UppPlus \delta_{\TripleNormBall_{\quasinorm}}}\np{0} }
                \intertext{by definition~\eqref{eq:Rockafellar-Moreau-subdifferential_a} of 
           the (Rockafellar-Moreau) subdifferential of a function}
    & \iff
      \SFM{\fonctionprimal}{\CouplingCapra}\np{\dual} 
      =\proscal{0}{\dual} \LowPlus \bp{ -\fonctionprimal\np{0} }
      \tag{by Proposition~\ref{pr:Fenchel-Moreau_conjugate_0-homogeneous}}
    \\
    & \iff
      \SFM{\fonctionprimal}{\CouplingCapra}\np{\dual} 
      =\CouplingCapra\np{0,\dual} \LowPlus \bp{ -\fonctionprimal\np{0} }
      \tag{since since \( \CouplingCapra\np{0,\dual} = 0 \) by~\eqref{eq:coupling_CAPRA}}
    \\
    & \iff
      \dual\in\subdifferential{\CouplingCapra}%
      { \fonctionprimal }\np{0} \tag{by definition~\eqref{eq:Capra-subdifferential_b} of
      the \Capra-subdifferential}
      \eqfinp
  \end{align*}

  Letting \( \sphere \in \TripleNormSphere_{\quasinorm} \) be such that
  \( \SFM{\fonctionprimal}{\CouplingCapra\Fenchelcoupling'}\np{\sphere}
  =\fonctionprimal\np{\sphere} \), we prove~\eqref{eq:Capra-subdifferential=Rockafellar-Moreau-subdifferential2} as follows: 
  \begin{align*}
    \dual \in
    \subdifferential{}{\SFM{\fonctionprimal}{\CouplingCapra\Fenchelcoupling'}}\np{\sphere} 
    \iff &
           \LFM{ \bp{ \SFM{\fonctionprimal}{\CouplingCapra\Fenchelcoupling'} } }\np{\dual}
           = \nscal{\sphere}{\dual} \LowPlus
           \bp{ -{\SFM{\fonctionprimal}{\CouplingCapra\Fenchelcoupling'}\np{\sphere} } }
           \tag{by~\eqref{eq:Rockafellar-Moreau-subdifferential_a}}
\\
    \iff &
           \SFM{\fonctionprimal}{\CouplingCapra}\np{\dual}
           = \nscal{\sphere}{\dual} \LowPlus
           \bp{ -{\SFM{\fonctionprimal}{\CouplingCapra\Fenchelcoupling'}\np{\sphere} } }
           \intertext{because the function
           \( \SFM{\fonctionprimal}{\CouplingCapra} = \LFM{ \bp{ \fonctionprimal \UppPlus
           \delta_{ \TripleNormSphere_{\quasinorm}^{\np{0}} } } }\) is a Fenchel conjugate
           by~\eqref{eq:Fenchel-Moreau_conjugate_0-homogeneous}, 
           hence is closed convex, and therefore equal to its Fenchel biconjugate
           \(
           \SFM{\fonctionprimal}{\CouplingCapra\Fenchelcoupling'\Fenchelcoupling}
           \eqfinv\)
            (see the background material in
  Appendix~\ref{Background_on_the_Fenchel_conjugacy})}
    \iff &
           \SFM{\fonctionprimal}{\CouplingCapra}\np{\dual}
           = \CouplingCapra\np{\sphere,\dual}  \LowPlus
           \bp{ -{\SFM{\fonctionprimal}{\CouplingCapra\Fenchelcoupling'}\np{\sphere} } }
           \intertext{by definition~\eqref{eq:coupling_CAPRA} of
           \( \CouplingCapra\np{\sphere,\dual} \) as \( \sphere \in
           \TripleNormSphere_{\quasinorm} \) hence
           \(\quasinorm\np{\primal} = 1 \) by~\eqref{eq:quasinorm_unit_sphere}, }
    \iff &
           \SFM{\fonctionprimal}{\CouplingCapra}\np{\dual}
           = \CouplingCapra\np{\sphere,\dual}  \LowPlus
           \bp{ -\fonctionprimal\np{\sphere} } 
           \tag{by assumption that \( \SFM{\fonctionprimal}{\CouplingCapra\Fenchelcoupling'}\np{\sphere}
  =\fonctionprimal\np{\sphere} \)}
    \\
    \iff &
           \dual \in
           \subdifferential{\CouplingCapra}{\fonctionprimal}\np{\sphere}
           \eqfinp
           \tag{by definition~\eqref{eq:Capra-subdifferential_b} of
           the \Capra-subdifferential}
  \end{align*}
  This ends the proof.
\end{proof}

Now, we are going to show how \Capra-couplings are a suitable tool to obtain lower
convex approximations of 0-homogeneous functions. 

\section{Best convex lower approximations of 0-homogeneous functions}
\label{sec:best_cvx_approx}

In~\S\ref{subsec:cvx_appr_0_hom}, we identify --- in terms of \Capra-conjugacy
introduced in Section~\ref{sec:capra} --- the best lower approximation of
0-homogeneous functions by convex and by positively 1-homogeneous convex
functions.
We apply these results to the \lzeropseudonorm \ in~\S\ref{subsec:appli_l0}.

\subsection{General result}
\label{subsec:cvx_appr_0_hom}

We prove that the best convex lower approximation of a 0-homogeneous function can be expressed
in term of its \Capra-conjugate~\eqref{eq:Fenchel-Moreau_conjugate_Fenchel}.
We also prove that the best positively 1-homogeneous closed convex lower
approximation of a 0-homogeneous function can be expressed 
in term of its \Capra-subdifferential~\eqref{eq:Capra-subdifferential_b}.
We recall that, for any subset \( \Dual\subset\RR^d \),
\( \supportFunction_{\Dual} : \RR^d \to \barRR\) denotes the 
\emph{support function of the subset~$\Dual$}:
\( \supportFunction_{\Dual}\np{\primal} = 
\sup_{\dual\in\Dual} \proscal{\primal}{\dual} \), for any
\( \primal \in \RR^d \).

The proof of the following theorem relies on results given in
Appendix~\ref{app:best_cvx_general}.

\begin{theorem}
  Let $\quasinorm : \RR^{d} \to \RR_{+}$ be a normalization function, with
  unit ``ball''~$\TripleNormBall_{\quasinorm}$ 
  defined in~\eqref{eq:quasinorm_unit_sphere}, and   
  with associated \Capra-coupling $\CouplingCapra$ in~\eqref{eq:coupling_CAPRA}. 
  Let \( \fonctionprimal : \RR^d \to \barRR \) be a 0-homogeneous
  function. Then, 
  \begin{enumerate}[topsep=0pt,itemsep=0pt,leftmargin=1cm]
  \item
    the function
  \( \SFM{\fonctionprimal}{\CouplingCapra\Fenchelcoupling'}
  \)
  is the tightest closed convex function below~\( \fonctionprimal \)
  on  the unit ``ball''~\( \TripleNormBall_{\quasinorm} \),
\item
  if \( \fonctionprimal\np{0}=0 \),
    the function
  \( \supportFunction_{ \subdifferential{\CouplingCapra}{\fonctionprimal}\np{0} }
  \) 
  is the tightest closed convex positively 1-homogeneous function below~\( \fonctionprimal \)
  on  the unit ``ball''~\( \TripleNormBall_{\quasinorm} \).
  \end{enumerate}
  \label{th:BestConvexLowerApproximations}
\end{theorem}

\begin{proof}
  From Proposition~\ref{prop:best_cvx_subset} with
  $\Uncertain = \TripleNormBall_{\quasinorm}$, the tightest closed convex
  function below $\fonctionprimal$ on the unit
  "ball"~$\TripleNormBall_{\quasinorm}$ is
  $\LFMbi{ \bp{ \fonctionprimal \UppPlus \delta_{\TripleNormBall_{\quasinorm}} }
  }$. As the function~$\fonctionprimal$ is 0-homogeneous, by
  Proposition~\ref{pr:Fenchel-Moreau_conjugate_0-homogeneous}, we have that
  \( \SFM{\fonctionprimal}{\CouplingCapra}= \LFM{ \bp{\fonctionprimal \UppPlus
      \delta_{\TripleNormBall_{\quasinorm}} } } \).  By taking the Fenchel
  conjugate~\eqref{eq:Fenchel_conjugate_reverse}, we deduce that
  \( \LFMbi{ \bp{
      \fonctionprimal \UppPlus \delta_{\TripleNormBall_{\quasinorm}} } }
  =\SFM{\fonctionprimal}{\CouplingCapra\Fenchelcoupling'} \).
  
  From Proposition~\ref{prop:best_pos_hom_cvx_subset} with
  $\Uncertain = \TripleNormBall_{\quasinorm}$, the tightest positively
  1-homogeneous closed convex function below $\fonctionprimal$ on the unit
  ``ball''~$\TripleNormBall_{\quasinorm}$ is
  $\supportFunction_{ \subdifferential{}{\np{\fonctionprimal \UppPlus
        \delta_{\TripleNormBall_{\quasinorm}}}}\np{0} }$. As the function~$\fonctionprimal$
  is 0-homogeneous, by Proposition~\ref{prop:capra_subdiff}, we have
  $\subdifferential{}{\np{\fonctionprimal \UppPlus
      \delta_{\TripleNormBall_{\quasinorm}}}}\np{0} =
  \subdifferential{\CouplingCapra}{\fonctionprimal}\np{0}$ and therefore
  $\supportFunction_{\subdifferential{}{\np{\fonctionprimal \UppPlus
        \delta_{\TripleNormBall_{\quasinorm}}}}\np{0}} =
  \supportFunction_{\subdifferential{\CouplingCapra}{\fonctionprimal}\np{0}}$,
  which gives the desired result.
\end{proof}

\subsection{Application to the \lzeropseudonorm}
\label{subsec:appli_l0}

As an application, we study the particular case of the \lzeropseudonorm. 
For any \( \primal \in \RR^d \) 
and subset \( K \subset \na{1,\ldots,d} \), 
we denote by
\( \primal_K \in \RR^d \) the vector which coincides with~\( \primal \),
except for the components outside of~$K$ that vanish.
Let $\TripleNorm{\cdot}$ be a norm on $\RR^{d}$, called the \emph{source norm}.
For any subset \( K \subset \ic{1,d} \),
we define the subspace
\(  \FlatRR_{K} = 
\bset{ \primal \in \RR^d }{ \primal_j=0 \eqsepv \forall j \not\in K } \) of~\(
\RR^d \). We introduce the \emph{$K$-restriction norm} \( \TripleNorm{\cdot}_{K} \), defined by  \( \TripleNorm{\primal}_{K} = \TripleNorm{\primal} \),
for any \( \primal \in \FlatRR_{K} \). We also define
the $\SetStar{K}$-norm
\( \TripleNorm{\cdot}_{K,\star} \), which is 
the norm \( \bp{\TripleNorm{\cdot}_{K}}_{\star} \),
given by the dual norm (on the subspace~\( \FlatRR_{K} \))
of the restriction norm~\( \TripleNorm{\cdot}_{K} \) 
to the subspace~\( \FlatRR_{K} \) (first restriction, then dual).
Following~\cite[Definition~3]{Chancelier-DeLara:2020_CAPRA_OPTIMIZATION}), for \( k \in \ic{1,d} \), we call \emph{generalized coordinate-$k$ norm}
the norm \( \CoordinateNorm{\TripleNorm{\cdot}}{k} \) 
whose dual norm is the 
\emph{generalized dual coordinate-$k$ norm}, denoted by
\( \CoordinateNormDual{\TripleNorm{\cdot}}{k} \), 
with expression $\CoordinateNormDual{\TripleNorm{\dual}}{k} =
\sup_{\cardinal{K} \leq k} \TripleNorm{\dual_K}_{K,\star}$ for all $\dual \in \RR^{d}$.
We denote by~\( \CoordinateNorm{\TripleNormBall}{k} \) and
\( \CoordinateNormDual{\TripleNormBall}{k} \)
the respective unit balls.
%
In Table~\ref{tab:Examples},
we give examples of coordinate-$k$ and dual coordinate-$k$ norms
in the case of the $\ell_p$ source norms.
We recall that, for $p\in [1,\infty]$, the dual norm of the $\ell_p$-norm~$\norm{\cdot}_{p}$
is the $\ell_q$-norm~$\norm{\cdot}_{q}$, where $q$ is such that \(1/p + 1/q=1\) 
(with the extreme cases $q=\infty$ when $p=1$, and $q=1$ when $p=\infty$).

\begin{table}[hbtp]
  \caption{Examples of coordinate-$k$ and dual coordinate-$k$ norms
    generated by the $\ell_p$ source norms 
    \( \TripleNorm{\cdot} = \norm{\cdot}_{p} \) for $p\in [1,\infty]$,
    where $1/p + 1/q =1$.
    For \( \dual \in \RR^d \), $\tau$ denotes a permutation of \( \{1,\ldots,d\} \) such that
    \( \module{ \dual_{\tau(1)} } \geq \module{ \dual_{\tau(2)} } 
    \geq \cdots \geq \module{ \dual_{\tau(d)} } \).
    \label{tab:Examples}}
  \centering
  \begin{tabular}{||c||c|c||}  
    \hline\hline 
    \small{source norm} \( \TripleNorm{\cdot} \) 
    & \( \CoordinateNorm{\TripleNorm{\cdot}}{k} \), $k\in\ic{1,d}$
    & \( \CoordinateNormDual{\TripleNorm{\cdot}}{k} \), $k\in\ic{1,d}$
    \\
    \hline\hline 
    \( \Norm{\cdot}_{p} \)
    & \lpsupportnorm{p}{k}, \;\;\( \LpSupportNorm{\primal}{p}{k} \) 
    & \lptopnorm{q}{k},\;\; \( \LpTopNorm{\dual}{q}{k} \) 
    \\
    & no analytic expression 
    & \(\LpTopNorm{\dual}{q}{k}=\bp{ \sum_{l=1}^{k} \module{ \dual_{\tau(l)} }^q }^{\frac{1}{q}} \)
    \\
    \hline
    \( \Norm{\cdot}_{1} \) 
    & \lpsupportnorm{1}{k} 
    & \lptopnorm{\infty}{k} 
    \\
    & $\ell_{1}$-norm
    & $\ell_{\infty}$-norm 
    \\[1mm]
    & \( \LpSupportNorm{\primal}{1}{k} = \Norm{\primal}_{1} \), $\forall k\in\ic{1,d}$
    & \( \LpTopNorm{\dual}{\infty}{k} = \Norm{\dual}_{\infty} \), $\forall k\in\ic{1,d}$
    \\
    \hline 
    \( \Norm{\cdot}_{2} \) 
    & \lpsupportnorm{2}{k} 
    & \lptopnorm{2}{k} 
    \\
    & \( \LpSupportNorm{\primal}{2}{k} \)
      no analytic expression
    & \( \LpTopNorm{\dual}{2}{k} = \sqrt{ \sum_{l=1}^{k} \module{ \dual_{\tau(l)} }^2 } \)
    \\
    & (computation \cite[Prop. 2.1]{Argyriou-Foygel-Srebro:2012})
    & 
    \\[1mm]
    & \( \LpSupportNorm{\primal}{2}{1} = \Norm{\primal}_{1} \)
    & \( \LpTopNorm{\dual}{2}{1}= \Norm{\dual}_{\infty} \)
    \\
    \hline 
    \( \Norm{\cdot}_{\infty} \)
    & \lpsupportnorm{\infty}{k} 
    & \lptopnorm{1}{k} 
    \\
    & \( \LpSupportNorm{\primal}{\infty}{k} =
      \max \na{ \frac{\Norm{\primal}_{1}}{k} , \Norm{\primal}_{\infty} } \) 
    & \( \LpTopNorm{\dual}{1}{k} = \sum_{l=1}^{k} \module{ \dual_{\tau(l)} } \)
    \\[1mm]
    & \( \LpSupportNorm{\primal}{\infty}{1} = \Norm{\primal}_{1} \)
    & \( \LpTopNorm{\dual}{1}{1}= \Norm{\dual}_{\infty} \)
    \\
    \hline\hline
  \end{tabular}
\end{table}

\paragraph{Best convex lower approximation of the \lzeropseudonorm\ on a unit
  ball}
Let $\TripleNorm{\cdot}$ be a source norm on $\RR^{d}$ and \( \varphi : \ic{0,d} \to \barRR \) be a function.
From Theorem~\ref{th:BestConvexLowerApproximations}, the function
\( \LFMr{ \bp{ \SFM{\np{ \varphi\circ\lzero}}{\CouplingCapra} } }
\)
is the tightest closed convex function below \( \varphi \circ
\lzero \) on  the unit ball~\( \TripleNormBall \), defined in~\eqref{eq:triplenorm_unit_sphere}.
We refer the reader to
\cite[Proposition~12]{Chancelier-DeLara:2020_CAPRA_OPTIMIZATION}, where 
several expressions of the function~\( \LFMr{ \bp{ \SFM{\np{
        \varphi\circ\lzero}}{\CouplingCapra} } } \) are provided.
In particular, it is shown that it is the largest closed convex function
below the integer valued function
 \( \CoordinateNorm{\TripleNormBall}{\LocalIndex} 
 \backslash \CoordinateNorm{\TripleNormBall}{\LocalIndex-1}
 \ni \primal \mapsto \varphi\np{\LocalIndex} \) 
          for $l\in\ic{1,d}$,
          and $\primal \in \CoordinateNorm{\TripleNormBall}{0} = \{0\} \mapsto \varphi\np{0}$, the function
          being infinite outside~\( \CoordinateNorm{\TripleNormBall}{d}=
          \TripleNormBall \).
          In dimension~2, this is seen in Figure~\ref{fig:graph1} which depicts
the tightest closed convex function below~\( \lzero \) on the Euclidean unit ball
on~$\RR^2$: the function goes up from zero to the value~1 on the border (sphere) of the lozenge
unit ball~$\TripleNormBall_{1}=\CoordinateNorm{\TripleNormBall}{1}$ of the
$\ell_1$-norm; then, from the lozenge to the value~2 on the border (sphere) of the round
unit ball~$\TripleNormBall_{2}=\CoordinateNorm{\TripleNormBall}{2}$ of the
$\ell_2$-norm; 
the function is discontinuous on the four ``sparse'' points on the
unit Euclidean circle, and takes the value~$+\infty$ outside the unit disk.
\begin{figure}
  \begin{center}
    {\includegraphics[width=5cm]{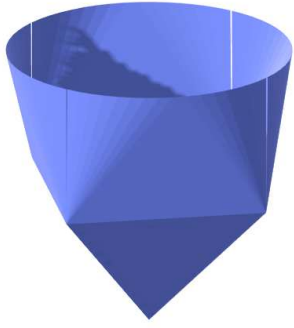}}
  \end{center}
  \caption{\label{fig:graph1}
    Tightest closed convex function below the \lzeropseudonorm\ on the Euclidean unit ball
    on~$\RR^2$}
\end{figure}

On the square unit ball~$\TripleNormBall_{\infty}$ of the $\ell_\infty$-norm, we obtain that
the best convex lower approximation on $\TripleNormBall_{\infty}$ is the
$\ell_{1}$-norm.
Indeed, if $\CouplingCapra$ is the \Capra-coupling~\eqref{eq:coupling_CAPRA}
associated with the $\ell_\infty$-norm, this best approximation is
\( \SFM{ \lzero}{\CouplingCapra\Fenchelcoupling'} \) by
Theorem~\ref{th:BestConvexLowerApproximations}.
Now, using~\cite[Proposition~12]{Chancelier-DeLara:2020_CAPRA_OPTIMIZATION},
we get that 
\( \SFM{ \lzero}{\CouplingCapra\Fenchelcoupling'} = \LFMr{
  \Bp{\sup_{\LocalIndex\in\ic{0,d}} \bc{
      \CoordinateNormDual{\TripleNorm{\cdot}}{\LocalIndex} -\LocalIndex} } } = \LFMr{
  \Bp{\sup_{\LocalIndex\in\ic{0,d}} \bc{
      \LpTopNorm{\cdot}{1}{\LocalIndex} -\LocalIndex} } } \)
(see Table~\ref{tab:Examples}).
As 
\( \sup_{\LocalIndex\in\ic{0,d}} \bc{
      \LpTopNorm{\dual}{1}{\LocalIndex} -\LocalIndex}  
= \sum_{i=1}^{d} \np{1 -  \abs{\dual_{i}}}\mathbf{1}_{\abs{\dual_{i}}\geq 1} \),
we get that 
\[
 \forall \primal \in \RR^d \eqsepv 
 \SFM{ \lzero}{\CouplingCapra\Fenchelcoupling'}\np{\primal} 
= \sum_{i=1}^{d} \sup_{\dual_i \in \RR}
    \bp{\primal_{i}\dual_{i} + \np{1 - \abs{\dual_{i}}}\mathbf{1}_{\abs{\dual_{i}}\geq 1}}
    = 
    \begin{cases}
      \Norm{\primal}_{1} \eqsepv & \primal \in \TripleNormBall_{\infty} \eqfinv \\
      + \infty \eqsepv &\text{otherwise.}
    \end{cases}
  \]

\paragraph{Best norm lower approximation of the \lzeropseudonorm\ on a unit ball}

We compute the expression of the best lower approximation of the
\lzeropseudonorm\ with a norm. Then, we show that for the $\ell_{p}$ source
norm, with $p\in [1,\infty]$, the tightest norm below $\lzero$ on the unit ball
$\TripleNormBall_{p}$ is the $\ell_{1}$-norm.

\begin{proposition}
  \label{prop:best_norm}
  Let $\TripleNorm{\cdot}$ be a source norm on~$\RR^d$,
  with associated sequence
  \( \sequence{\CoordinateNorm{\TripleNorm{\cdot}}{\LocalIndex}}{\LocalIndex\in\ic{1,d}} \)
  of coordinate-$k$ norms and sequence
  \( \sequence{\CoordinateNorm{\TripleNormDual{\cdot}}{\LocalIndex}}{\LocalIndex\in\ic{1,d}} \)
  of dual coordinate-$k$ norms
  and their respective unit balls~\( \sequence{\CoordinateNormDual{\TripleNormBall}{\LocalIndex}}{\LocalIndex\in\ic{1,d}} \).
  Let \( \varphi : \ic{0,d} \to \RR_+ \cup \na{+\infty} \) be a function
  such that \( \varphi\np{\LocalIndex} > \varphi\np{0}=0 \) for all \(
  \LocalIndex\in\ic{1,d} \), and with
  \( \varphi\np{\LocalIndex} <+\infty \) for at least one \(
  \LocalIndex\in\ic{1,d} \). 
  Then, there exists a 
  norm~\( \CoordinateNorm{\TripleNorm{\cdot}}{\varphi} \) on~$\RR^d$, characterized by its dual norm \( \CoordinateNormDual{\TripleNorm{\cdot}}{\varphi} \), which has unit ball
  \( \CoordinateNormDual{\TripleNormBall}{\varphi} = \bigcap_{ \LocalIndex\in\ic{1,d} } \varphi\np{\LocalIndex} \CoordinateNormDual{\TripleNormBall}{\LocalIndex} \),
  and $\CoordinateNorm{\TripleNorm{\cdot}}{\varphi} = \supportFunction_{ \CoordinateNormDual{\TripleNormBall}{\varphi} }$,
  where by convention $+\infty \CoordinateNormDual{\TripleNormBall}{\LocalIndex} = \RR^{d}$. The norm~\( \CoordinateNorm{\TripleNorm{\cdot}}{\varphi} \) 
  is the tightest norm below \( \varphi \circ
  \lzero \) on the unit ball~\( \TripleNormBall \), that is,
  \begin{equation}
    \underbrace{ \CoordinateNorm{\TripleNorm{\primal}}{\varphi}}%
    _{ \textrm{tightest norm} } 
    \leq \varphi\bp{\lzero\np{\primal}}
    \eqsepv \forall \primal \in \TripleNormBall 
    \eqfinp
  \end{equation}
  %
  
\end{proposition}

\begin{proof}
  From~\cite[Proposition 14]{Chancelier-DeLara:2020_CAPRA_OPTIMIZATION},
  we have that $\subdifferential{\CouplingCapra}{\np{\varphi\circ\lzero}}\np{0}
  = \CoordinateNormDual{\TripleNormBall}{\varphi}$. With
  Theorem~\ref{th:BestConvexLowerApproximations}, we deduce that
  $\supportFunction_{ \CoordinateNormDual{\TripleNormBall}{\varphi} } =
  \CoordinateNorm{\TripleNorm{\cdot}}{\varphi} $ is the tightest closed convex
  positively 1-homogeneous function below $\varphi\circ\lzero$ on the unit ball
  $\TripleNormBall$. We can easily check that $\supportFunction_{
    \CoordinateNormDual{\TripleNormBall}{\varphi} }$ indeed defines a norm. 
  The norm~\( \CoordinateNorm{\TripleNorm{\cdot}}{\varphi} \)
  was introduced in
  \cite[Proposition~15]{Chancelier-DeLara:2020_CAPRA_OPTIMIZATION}
  under the slightly stronger assumption that
  \( \varphi\np{\LocalIndex} <+\infty \) for all
  \( \LocalIndex\in\ic{1,d} \).
\end{proof}

\begin{proposition} (Application with the $\ell_{p}$ source norm).
  \label{prop:appli_lp_source}
  Let $\TripleNorm{\cdot} = \Norm{\cdot}_{p}$, with $p \in [1, +\infty]$, be the source norm and let \( \varphi : \ic{0,d} \to \RR_+ \cup \na{+\infty} \) with \( \varphi\np{\LocalIndex} > \varphi\np{0}=0 \) for all \(
  \LocalIndex\in\ic{1,d} \) and
  \( \varphi\np{\LocalIndex} <+\infty \) for at least one \(
  \LocalIndex\in\ic{1,d} \). We also assume that either, 
for $p \in ]1, +\infty]$ and $q$ such that $1/p + 1/q = 1$, 
the function \( \ic{1, d} \ni \LocalIndex \mapsto
\varphi\np{\LocalIndex}^q/\LocalIndex \) is nondecreasing, 
or, for $p = 1$ and $q = +\infty$, the function~$\varphi$ is nondecreasing.
  Then, the tightest norm below $\varphi\circ\lzero$ on the unit ball $\TripleNormBall_{p}$ is \( \CoordinateNorm{\TripleNorm{\cdot}}{\varphi} =  \varphi\np{1}\Norm{\cdot}_{1}\). In particular, the $\ell_{1}$-norm is the tightest norm below $\lzero$ on the unit ball $\TripleNormBall_{p}$ of any $\ell_p$-norm, for \( p \in [1,\infty] \).
\end{proposition}

\begin{proof}
  From Proposition~\ref{prop:best_norm}, the tightest norm below
  $\varphi\circ\lzero$ on the unit ball $\TripleNormBall_{p}$ is $
  \CoordinateNorm{\TripleNorm{\cdot}}{\varphi}$ and, for all $\dual \in
  \RR^{d}$, we have that \( \CoordinateNormDual{\TripleNorm{\dual}}{\varphi}
    =     \sup_{ \LocalIndex\in\ic{1,d} }
    \frac{\CoordinateNormDual{\TripleNorm{\dual}}{\LocalIndex}}%
    { \varphi\np{\LocalIndex} } \)   
(see \cite[Proposition~15]{Chancelier-DeLara:2020_CAPRA_OPTIMIZATION}).
%
  The source norm is \( \Norm{\cdot}_{p} \). If $p = 1$, then from
  Table~\ref{tab:Examples},
  \(\CoordinateNormDual{\TripleNorm{\dual}}{\LocalIndex} =
  \LpTopNorm{\dual}{\infty}{\LocalIndex} = \Norm{\dual}_{\infty} \). We deduce
  that
  \( \CoordinateNormDual{\TripleNorm{\dual}}{\varphi} = \sup_{
    \LocalIndex\in\ic{1,d} } \frac{\Norm{\dual}_{\infty}}%
  { \varphi\np{\LocalIndex} } = \frac{\Norm{\dual}_{\infty}}{\varphi\np{1}} \)
  as $\varphi$ is nondecreasing. For $p > 1$, from Table~\ref{tab:Examples}, we
  get that
  \(\CoordinateNormDual{\TripleNorm{\dual}}{\LocalIndex} =
  \LpTopNorm{\dual}{q}{\LocalIndex}=\bp{ \sum_{l=1}^{\LocalIndex} \module{
      \dual_{\tau(l)} }^q }^{\frac{1}{q}} \).
  Using that
  $\LocalIndex \mapsto \varphi\np{\LocalIndex}^q/\LocalIndex$ is nondecreasing,
  it can be computed that 
  \( \LocalIndex \mapsto
  \LpTopNorm{\dual}{q}{\LocalIndex}/\varphi\np{\LocalIndex} \) is nonincreasing
  as follows:
  \[
\Bp{\frac{\LpTopNorm{\dual}{q}{\LocalIndex+1}}{\varphi\np{\LocalIndex+1}}}^q
      =
      \frac{\np{\sum_{l=1}^{\LocalIndex} \module{ \dual_{\tau(l)} }^q} 
      +  \module{ \dual_{\tau(\LocalIndex+1)}}^q}{\varphi\np{\LocalIndex+1}^{q}}
\le 
      \np{1+\frac{1}{j}}\frac{\bp{\sum_{l=1}^{\LocalIndex} \module{ \dual_{\tau(l)} }^q} 
      }{\varphi\np{\LocalIndex+1}^{q}}
\le
      \frac{\np{\sum_{l=1}^{\LocalIndex} \module{ \dual_{\tau(l)} }^q} 
      }{\varphi\np{\LocalIndex}^{q}}
      =\Bp{\frac{\LpTopNorm{\dual}{q}{\LocalIndex}}{\varphi\np{\LocalIndex}}}^q
      \eqfinp
    \]
%
  Therefore, \(        \CoordinateNormDual{\TripleNorm{\dual}}{\varphi}
  =  \sup_{ \LocalIndex\in\ic{1,d} }
  \frac{\CoordinateNormDual{\TripleNorm{\dual}}{\LocalIndex}}%
  { \varphi\np{\LocalIndex} } = \frac{\LpTopNorm{\dual}{q}{1}}{\varphi\np{1}} =
  \frac{\Norm{\dual}_{\infty}}{\varphi\np{1}} \), as \( \LpTopNorm{\dual}{q}{1}
  = \module{ \dual_{\tau(1)} }= \Norm{\dual}_{\infty} \). 
  Thus, for $p \in [1, \infty]$, we have obtained
  $\CoordinateNormDual{\TripleNorm{\dual}}{\varphi} =
  \frac{\Norm{\dual}_{\infty}}{\varphi\np{1}}$ from which we deduce that
  \( \CoordinateNorm{\TripleNorm{\cdot}}{\varphi} =
  \varphi\np{1}\Norm{\cdot}_{1}\).
  The last statement of the theorem follows by taking $\varphi$ being the identity mapping. This concludes the proof.
\end{proof}


Fazel~\cite[Theorem 1, \S5.1.4]{Fazel:2002matrix} shows that the best convex lower
approximation of the rank function
on the spectral norm unit ball
is given by the nuclear norm.
By considering singular values, we easily deduce that the best convex lower
approximation of the \lzeropseudonorm\ on the $\ell_{\infty}$ unit ball
is given by the $\ell_{1}$-norm.
Thus, in a sense, Proposition~\ref{prop:appli_lp_source} generalizes the result of
Fazel. 

\section{Conclusion}

In this paper, we have extended the 
\Capra\ couplings and conjugacies introduced
in~\cite{Chancelier-DeLara:2020_CAPRA_OPTIMIZATION,Chancelier-DeLara:2021_ECAPRA_JCA}.
Indeed, they now depend on a given normalization function (i.e. a norm without the requirement for
subadditivity). With this new coupling, we are able to provide expressions for
the best convex (and positively 1-homogeneous) lower approximations of
0-homogeneous functions on the unit ``ball'' of the normalization function. As an
application, we show that the best norm lower approximation of the
\lzeropseudonorm \ on any $\ell_{p}$ unit ball, $p \in [1, +\infty]$, is the
$\ell_{1}$ norm, thus strengthening theoretical grounds for the use of the
$\ell_{1}$ norm as a sparsity-inducing term in optimization problems.
We have also provided the tightest convex lower convex
  approximation of the l0 pseudonorm on the unit ball of any norm.

\appendix

\section{Background on the Fenchel conjugacy}
\label{Background_on_the_Fenchel_conjugacy}

We review concepts and notations related to the Fenchel conjugacy
(we refer the reader to \cite{Rockafellar:1974}).
For any function \( \fonctionuncertain : \RR^d \to \barRR \),
its \emph{epigraph} is \( \epigraph\fonctionuncertain= 
\defset{ \np{\uncertain,t}\in\RR^d\times\RR}%
{\fonctionuncertain\np{\uncertain} \leq t} \),
its \emph{effective domain} is 
\( \dom\fonctionuncertain= 
\defset{\uncertain\in\RR^d}{ \fonctionuncertain\np{\uncertain} <+\infty}
\).
A function \( \fonctionuncertain : \RR^d \to \barRR \)
is said to be
\emph{convex} if its epigraph is a convex set,
\emph{proper} if it never takes the value~$-\infty$
and that \( \dom\fonctionuncertain \not = \emptyset \),
\emph{lower semi continuous (lsc)} if its epigraph is closed,
\emph{closed} if it is either lsc and nowhere having the value $-\infty$,
or is the constant function~$-\infty$ \cite[p.~15]{Rockafellar:1974}.
Closed convex functions are the two constant functions~$-\infty$ and~$+\infty$
united with all proper convex lsc functions.\footnote{%
  In particular, any closed convex function that takes at least one finite value
  is necessarily proper convex~lsc. \label{ft:closed_convex_function}}
It is proved that the Fenchel conjugacy
(indifferently 
\( \fonctionprimal \mapsto \LFM{\fonctionprimal} \)
or
\( \fonctiondual \mapsto \LFMr{\fonctiondual} \) as below)
induces a one-to-one correspondence
between the closed convex functions on~$\RR^d$ and themselves
\cite[Theorem~5]{Rockafellar:1974}.
For any functions \( \fonctionprimal : \RR^d  \to \barRR \)
and \( \fonctiondual : \RR^d \to \barRR \), 
we denote\footnote{%
  In convex analysis, one does not use
  the notation~\( \LFMr{} \)
  in~\eqref{eq:Fenchel_conjugate_reverse} and~\( \LFMbi{} \) in~\eqref{eq:Fenchel_biconjugate},
  but simply~\( \LFM{} \) and~\( ^{\Fenchelcoupling\Fenchelcoupling} \).
  We use~\( \LFMr{} \) and~\( \LFMbi{} \) to be consistent with the notation~\eqref{eq:Fenchel-Moreau_biconjugate} for general conjugacies.}
\begin{subequations}
  \begin{align}
    \LFM{\fonctionprimal}\np{\dual} 
    &= 
      \sup_{\primal \in \RR^d} \Bp{ \proscal{\primal}{\dual} 
      \LowPlus \bp{ -\fonctionprimal\np{\primal} } } 
      \eqsepv \forall \dual \in \RR^d
      \eqfinv
      \label{eq:Fenchel_conjugate}
    \\
    \LFMr{\fonctiondual}\np{\primal} 
    &= 
      \sup_{ \dual \in \DUAL } \Bp{ \proscal{\primal}{\dual} 
      \LowPlus \bp{ -\fonctiondual\np{\dual} } } 
      \eqsepv \forall \primal \in \RR^d
      \eqfinv
      \label{eq:Fenchel_conjugate_reverse}
    \\
    \LFMbi{\fonctionprimal}\np{\primal} 
    &= 
      \sup_{\dual \in \RR^d} \Bp{ \proscal{\primal}{\dual} 
      \LowPlus \bp{ -\LFM{\fonctionprimal}\np{\dual} } } 
      \eqsepv \forall \primal \in \RR^d
      \eqfinp
      \label{eq:Fenchel_biconjugate}
  \end{align}
  The Fenchel biconjugate of a 
  function \( \fonctionprimal : \RR^d  \to \barRR \) satisfies
  \begin{equation}
    \LFMbi{\fonctionprimal}\np{\primal}
    \leq \fonctionprimal\np{\primal}
    \eqsepv \forall \primal \in \RR^d 
    \eqfinp
    \label{eq:Fenchel_biconjugate_leq_fonctionprimal}
  \end{equation}
\end{subequations}
In \cite[p.~214-215]{Rockafellar:1970},
the notions\footnote{%
  See the historical note in \cite[p.~343]{Rockafellar-Wets:1998}.} of (Moreau) subgradient
and of (Rockafellar) subdifferential are defined for a convex function.
Following the definition of the subdifferential
of a function with respect to a duality 
in \cite{Akian-Gaubert-Kolokoltsov:2002},
we define the \emph{(Rockafellar-Moreau) subdifferential} \( \subdifferential{}{\fonctionprimal}\np{\primal} \)
of a function \( \fonctionprimal : \RR^d  \to \barRR \)
at~\( \primal \in \RR^d \) by
  \begin{equation}
    \subdifferential{}{\fonctionprimal}\np{\primal}
    =    \defset{ \dual \in \RR^d }{ %
      \LFM{\fonctionprimal}\np{\dual} 
      = \nscal{\primal}{\dual} 
      \LowPlus \bp{ -\fonctionprimal\np{\primal} } }
    \eqfinp
    \label{eq:Rockafellar-Moreau-subdifferential_a}
  \end{equation}
  When the function~\( \fonctionprimal \) is proper convex and
  \( \primal \in  \dom\fonctionprimal \), we recover the classic definition.

\section{Best convex lower approximations of a function on a subset}
\label{app:best_cvx_general}

We compute the best convex (and convex positively 1-homogeneous) lower
approximations  of a general function on a (not necessarily convex) subset $\Uncertain \subset \RR^{d}$. 

\begin{proposition}
  \label{prop:best_cvx_subset}
  For any subset \( \Uncertain \subset \RR^d \) and any 
  function \( \fonctionprimal : \RR^d \to \barRR \), the best closed convex lower approximation $\hat{\fonctionprimal}$ of $\fonctionprimal$ on $\Uncertain$ is given by 
  $\hat{\fonctionprimal} =
  \LFMbi{ \bp{ \fonctionprimal \UppPlus \delta_{\Uncertain} } }$.
\end{proposition}

\begin{proof}
  For any subset \( \Uncertain \subset \RR^d \) and any 
  function \( \fonctionprimal : \RR^d \to \barRR \), 
  we define the set of functions 
  \begin{equation}
    \BestConvexLowerApproximations{\fonctionprimal}{\Uncertain}=
    \bset{ \textrm{closed convex function }
      \fonctionprimalbis: \RR^d \to \barRR}%
    { \fonctionprimalbis\np{\primal}\leq \fonctionprimal\np{\primal}
      \eqsepv \forall \primal \in \Uncertain }.
    \label{eq:BestConvexLowerApproximations}
  \end{equation}
  As the set $\mathcal{F} = \ba{\textrm{function } \fonctionprimal: \RR^{d} \to
    \barRR}$ endowed with the partial order~$\leq$ is a complete lattice, the
  subset $\BestConvexLowerApproximations{\fonctionprimal}{\Uncertain} \subset
  \mathcal{F}$ has a (unique) supremum \( \hat{\fonctionprimal} =
  \bigvee\BestConvexLowerApproximations{\fonctionprimal}{\Uncertain} \). As the
  set of closed convex functions is stable by pointwise supremum, the
  function~\( \hat{\fonctionprimal} : \RR^{d} \to \barRR \) is given, for all
  $\primal \in \Uncertain$, by $\hat{\fonctionprimal}\np{\primal}=
  \sup\bset{\fonctionprimalbis\np{\primal}}%
  {\fonctionprimalbis\in\BestConvexLowerApproximations{\fonctionprimal}{\Uncertain}}$.

  By~\eqref{eq:Fenchel_biconjugate_leq_fonctionprimal},
  it is easily established that $\LFMbi{ \bp{ \fonctionprimal \UppPlus \delta_{\Uncertain} } } \in \BestConvexLowerApproximations{\fonctionprimal}{\Uncertain}$. For any $\fonctionprimalbis \in
  \BestConvexLowerApproximations{\fonctionprimal}{\Uncertain}$, we have
  $\fonctionprimalbis  \UppPlus \delta_{\closedconvexhull\Uncertain} \in \BestConvexLowerApproximations{\fonctionprimal}{\Uncertain}$ and $\fonctionprimalbis  \UppPlus \delta_{\closedconvexhull\Uncertain} 
  \leq \LFMbi{ \bp{ \fonctionprimal \UppPlus \delta_{\Uncertain} } }$,
  as \( \delta_{\closedconvexhull\Uncertain} \leq \delta_{\Uncertain} \)
  where $\closedconvexhull\Uncertain$ denotes the closed convex hull of the set $\Uncertain$.
  As \( \fonctionprimalbis \leq \fonctionprimalbis \UppPlus \delta_{\closedconvexhull\Uncertain} \),
  we conclude that
  \( \hat{\fonctionprimal} =
  \bigvee\BestConvexLowerApproximations{\fonctionprimal}{\Uncertain}
  \leq \LFMbi{ \bp{ \fonctionprimal \UppPlus \delta_{\Uncertain} } }     
  \), hence that \(    \hat{\fonctionprimal} =
  \LFMbi{ \bp{ \fonctionprimal \UppPlus \delta_{\Uncertain} } } \)
  since \( \LFMbi{ \bp{ \fonctionprimal \UppPlus \delta_{\Uncertain} } } \in
  \BestConvexLowerApproximations{\fonctionprimal}{\Uncertain} \).
\end{proof}

\begin{remark}
  We cannot replace $\delta_{\Uncertain}$
  by~$\delta_{\closedconvexhull\Uncertain}$ in the statement of Proposition~\ref{prop:best_cvx_subset}, 
  although $\delta_{\closedconvexhull\Uncertain}$ is involved in the proof. For
  instance, consider $\fonctionprimal = \abs{\cdot}$ being the absolute value
  function defined on $\RR$ and $\Uncertain = ]-\infty, -1] \cup [1,+\infty[$,
  hence $\closedconvexhull\Uncertain = \RR$. We can easily check that $\LFMbi{
    \bp{ \fonctionprimal \UppPlus \delta_{\Uncertain} } }\np{\primal} = \abs{x}$
  if $\primal \in \Uncertain$ and $\LFMbi{ \bp{ \fonctionprimal \UppPlus
      \delta_{\Uncertain} } }\np{\primal} = 1$ if $\primal \in \RR \setminus
  \Uncertain$, that is, $\hat{\fonctionprimal} =\max\na{1, \abs{\cdot}}$. 
  On the other hand, we have $\LFMbi{ \bp{ \fonctionprimal \UppPlus
      \delta_{\closedconvexhull\Uncertain} } } = \LFMbi{ \fonctionprimal } =
  \abs{\cdot}$,
  and therefore $\LFMbi{ \bp{ \fonctionprimal \UppPlus
      \delta_{\closedconvexhull\Uncertain} } } = \abs{\cdot}
  \neq \max\na{1, \abs{\cdot}}=\LFMbi{ \bp{ \fonctionprimal \UppPlus \delta_{\Uncertain} } }$.
\end{remark}

Now, we give the best approximation of a general function $\fonctionprimal: \RR^{d} \to \barRR$ with a closed convex positively 1-homogeneous function.

\begin{proposition}
  For any subset \( \Uncertain \subset \RR^d \) with $0 \in \Uncertain$ and any function \( \fonctionprimal : \RR^d \to \barRR \) such that $\fonctionprimal\np{0} = 0$, 
  the best closed convex, positively 1-homogeneous approximation $\tilde{\fonctionprimal}$ of $\fonctionprimal$ on $\Uncertain$ is given by
  $\tilde{\fonctionprimal} =
  \supportFunction_{ \subdifferential{}{\np{\fonctionprimal \UppPlus \delta_{\Uncertain}}}\np{0} } 
  \eqfinp $
  \label{prop:best_pos_hom_cvx_subset}
\end{proposition}

\begin{proof}
  For any subset \( \Uncertain \subset \RR^d \) and any 
  function \( \fonctionprimal : \RR^d \to \barRR \) such that \( \fonctionprimal\np{0}=0 \), 
  we define the set of functions 
  \begin{equation}
    \BestConvexLowerHomogeneous{\fonctionprimal}{\Uncertain}= \bset{ \textrm{closed convex, positively 1-homogeneous function }
      \fonctionprimalbis: \RR^{d} \to \barRR}%
    { \fonctionprimalbis\np{\primal}\leq \fonctionprimal\np{\primal}
      \eqsepv \forall \primal \in \Uncertain }
    \eqfinp 
    \label{eq:BestConvexLowerHomogeneous}
  \end{equation}
  As the set $\mathcal{F} = \ba{\textrm{function } \fonctionprimal: \RR^{d} \to
    \barRR}$ endowed with the partial order~$\leq$ is a complete lattice, the
  subset $\BestConvexLowerHomogeneous{\fonctionprimal}{\Uncertain} \subset
  \mathcal{F}$ has a (unique) supremum \( \tilde{\fonctionprimal} =
  \bigvee\BestConvexLowerHomogeneous{\fonctionprimal}{\Uncertain} \). As the set
  of closed convex, positively 1-homogeneous functions is stable by pointwise
  supremum, the function~\( \tilde{\fonctionprimal} : \RR^{d} \to \barRR \) is
  given, for all $\primal \in \Uncertain$, by
  $\tilde{\fonctionprimal}\np{\primal}=
  \sup\bset{\fonctionprimalbis\np{\primal}}%
  {\fonctionprimalbis\in\BestConvexLowerHomogeneous{\fonctionprimal}{\Uncertain}}$.
  From~\cite[Theorem 8.24]{Rockafellar-Wets:1998}, the proper closed convex positively 1-homogeneous functions can be identified with the support functions of nonempty closed convex subsets of $\RR^{d}$. Thus, we describe the functions of $\BestConvexLowerHomogeneous{\fonctionprimal}{\Uncertain}$ by means of support functions. 
  Let \( \Dual \subset \RR^d \) be a nonempty closed convex set such that
  $\supportFunction_{\Dual}\np{\primal} \leq \fonctionprimal\np{\primal}$, for all $\primal \in \Uncertain$.
  We have
  \begin{align*}
    \dual\in\Dual
    & \implies
      \proscal{\primal}{\dual} \leq \fonctionprimal\np{\primal}
      \eqsepv \forall \primal \in \Uncertain
      \tag{by definition 
      of the support function~$\supportFunction_{\Dual}$} 
    \\
    & \iff
      \proscal{\primal}{\dual} \LowPlus \bp{-\fonctionprimal\np{\primal}}
      \leq
      0=\proscal{0}{\dual} \LowPlus \bp{-\fonctionprimal\np{0}}
      \eqsepv \forall \primal \in \Uncertain
      \intertext{using property of the Moreau lower addition \cite{Moreau:1970},
and the assumption that \( \fonctionprimal\np{0}=0 \)}
    & \iff
      \proscal{\primal}{\dual} \LowPlus \bp{ -\np{\fonctionprimal
      \UppPlus \delta_{\Uncertain}}\np{\primal} }
      \leq
      \proscal{0}{\dual} \LowPlus \bp{ -\np{\fonctionprimal
      \UppPlus \delta_{\Uncertain}}\np{0} }
      \eqsepv \forall \primal \in \RR^d \tag{as $0 \in \Uncertain$} \eqfinv
    \\
    & \iff
      \dual\in\subdifferential{}%
      { \bp{\fonctionprimal \UppPlus \delta_{\Uncertain}} }\np{0}
  \end{align*}
    by definition~\eqref{eq:Rockafellar-Moreau-subdifferential_a} of 
                    the (Rockafellar-Moreau) subdifferential.                   
  Thus, we have obtained that $\Dual \subset \subdifferential{}%
  { \bp{\fonctionprimal \UppPlus \delta_{\Uncertain}} }\np{0}$ from which we deduce the inequality $\supportFunction_{\Dual} \leq \supportFunction_{\subdifferential{}%
    { \np{\fonctionprimal \UppPlus \delta_{\Uncertain}} }\np{0}}$.
We get that $\tilde{\fonctionprimal} \leq \supportFunction_{\subdifferential{}%
    { \np{\fonctionprimal \UppPlus \delta_{\Uncertain}} }\np{0}}$, as the subset 
  $\bset{\supportFunction_{\Dual}, \Dual \subset \RR^{d} \text{ nonempty closed convex}}{\supportFunction_{\Dual}\np{\primal} \leq \fonctionprimal\np{\primal}
    \eqsepv \forall \primal \in \Uncertain}$ is equal
  to~$\BestConvexLowerHomogeneous{\fonctionprimal}{\Uncertain}$
  in~\eqref{eq:BestConvexLowerHomogeneous}. 

  On the other hand, using the previous equivalences, if
  $\dual\in\subdifferential{} { \bp{\fonctionprimal \UppPlus
      \delta_{\Uncertain}} }\np{0}$, we get that $\proscal{\primal}{\dual} \leq
  \fonctionprimal\np{\primal}$ for all $\primal \in \Uncertain$. Therefore, we
  obtain that 
  $\supportFunction_{\subdifferential{}%
    { \np{\fonctionprimal \UppPlus \delta_{\Uncertain}} }\np{0}}\np{\primal} = \sup_{\dual \in \subdifferential{}%
    { \np{\fonctionprimal \UppPlus \delta_{\Uncertain}} }\np{0}} \proscal{\primal}{\dual} \leq \fonctionprimal\np{\primal}$ for all $\primal \in \Uncertain$,
  that is, $\supportFunction_{\subdifferential{}%
    { \np{\fonctionprimal \UppPlus \delta_{\Uncertain}} }\np{0}} \in \BestConvexLowerHomogeneous{\fonctionprimal}{\Uncertain}$. We conclude that $\tilde{\fonctionprimal} =
  \supportFunction_{ \subdifferential{}{\np{\fonctionprimal \UppPlus \delta_{\Uncertain}}}\np{0} }$.
\end{proof}

\newcommand{\noopsort}[1]{} \ifx\undefined\allcaps\def\allcaps#1{#1}\fi

\end{document}